\documentclass[12pt]{amsart}

\usepackage{amsmath,enumitem}
\usepackage[english,activeacute]{babel}
\usepackage{mathrsfs} 
\usepackage[latin1]{inputenc}
\usepackage{amssymb}
\usepackage{amsthm}
\usepackage{graphics}
\usepackage{array}
\usepackage{pdflscape}
\usepackage{a4wide}
\usepackage{url}
\usepackage{float}
\usepackage{hyperref}

\usepackage[dvipsnames]{xcolor}

\hypersetup{
    pdftoolbar=true,        % show Acrobat toolbar?
    pdfmenubar=true,        % show Acrobat menu?
    pdffitwindow=false,     % window fit to page when opened
    pdfstartview={FitH},    % fits the width of the page to the window
    pdftitle={},
    pdfauthor={},
    pdfsubject={},
    pdfkeywords={},
    pdfnewwindow=true,      % links in new window
    colorlinks=true,       % false: boxed links; true: colored links
    linkcolor=brown,          % color of internal links
    citecolor=brown,        % color of links to bibliography
    filecolor=brown,      % color of file links
    urlcolor=brown,           % color of external links
    unicode=true          % non-Latin characters in Acrobat's bookmarks
}

\theoremstyle{plain}
\newtheorem{theorem}{Theorem}

\newtheorem{lemma}[theorem]{Lemma}
\newtheorem{proposition}[theorem]{Proposition}

\theoremstyle{definition}
\newtheorem{definition}[theorem]{Definition}
\newtheorem{example}[theorem]{Example}
\newtheorem{remark}[theorem]{Remark}
\newtheorem{note}[theorem]{Note}

\def\piros#1{{\color{red}#1}}%
\def\kek#1{{\color{blue}#1}}%

\DeclareMathOperator{\im}{im}

\title[Combinatorial proof of an identity on Genocchi numbers]{Combinatorial proof of an identity on Genocchi numbers}

\author{Be\'ata B\'enyi}
\address{\noindent Faculty of Water Sciences, University of Public Service, Baja, HUNGARY}
\email{benyi.beata@uni-nke.hu}

\author{Matthieu Josuat-Verg\`{e}s}
\address{\noindent CNRS, IRIF (UMR 8243), Universit\'e de Paris, FRANCE}
\email{matthieu.josuat-verges@u-pem.fr}

\date{\today}
\subjclass[2010]{05A19, 11B68}
\keywords{Genocchi numbers, poly-Bernoulli numbers, combinatorial identity, bijections, Callan sequences}

\begin{document}
\begin{abstract}In this note we present a combinatorial proof of an identity involving poly-Bernoulli numbers and Genocchi numbers. 
We introduce the combinatorial objects, $m-$barred Callan sequences and show that the identity holds in a more general manner.
\end{abstract}

%--------------------------------------------------------

\maketitle

%%%%%%%%%%%%%%%%%%
\section{Introduction}
%%%%%%%%%%%%%%%%%%

In this paper, we present a combinatorial proof of an identity that establishes an interesting relation between the $C$-poly-Bernoulli numbers and Genocchi numbers.
The Genocchi numbers $G_n$ can be defined for instance by the generating function \cite[Exercise~5.8]{Stanley}:
\begin{align*}
  \sum_{n=0}^{\infty}G_n\frac{t^n}{n!}=\frac{2t}{e^t+1}.
\end{align*}
The first values are 
\[
  0,1,-1,0,1,0,-3,0,17,0,-155,\ldots
\] 
(see \href{http://oeis.org/A036968}{A036968} in \cite{OEIS}).

The {\it Poly-Bernoulli numbers} $B_n^{(k)}$ are defined for $k \in \mathbb{Z}$ by 
\begin{align}\label{pBdef}
\frac{\mbox{Li}_k(1-e^{-t})}{1-e^{-t}}=\sum_{n=0}^{\infty}B_n^{(k)}\frac{t^n}{n!},
\end{align}
where $\mbox{Li}_k(z)$ is the polylogarithm function given by
\[
  \mbox{Li}_k(z)=\sum_{m=1}^{\infty}\frac{z^m}{m^k}\quad (|z|<1).
\]
For  negative $k$ indices, the poly-Bernoulli numbers are integers (\href{http://oeis.org/A099594}{A099594} in~\cite{OEIS}). Several combinatorial objects are enumerated by these numbers, for instance lonesum matrices~\cite{Brew}, matrices uniquely reconstructible from their row and column sum vectors. For further objects see \cite{BH15, BH17}. It is proven in \cite{AK99} analytically and in \cite{BH15} combinatorially that 
\begin{align}\label{pB-id}
\sum_{j=0}^n (-1)^j B_{n-j}^{(-j)}=0.
\end{align}
Arakawa - Kaneko introduced a $C$-version of poly-Bernoulli numbers in \cite{AK99A} as
\begin{align}\label{pBcdef}
\frac{\mbox{Li}_k(1-e^{-t})}{e^{t}-1}=\sum_{n=0}^{\infty}C_n^{(k)}\frac{t^n}{n!}.
\end{align}
In particular, $C^{(1)}_n$ is the ordinary Bernoulli numbers $B_n$ defined by 
\[\frac{t}{e^t-1}=\sum_{n=0}^{\infty}B_n\frac{t^n}{n!}.\]
For negative $k$, these numbers $C_n^{(k)}$ are again positive integers.  They are the array \href{http://oeis.org/A136126}{A136126} in \cite{OEIS}, see also Table~\ref{C-numbers} below.  In \cite{BH17}, this integer array is introduced as the number of lonesum matrices without any all-zero columns.

Kaneko - Sakurai - Tsumura \cite{KST18} proved the $C$-version of the identity \eqref{pB-id} using analytical methods.
\begin{align}\label{pBc-id}
\sum_{j=0}^{n}(-1)^{j}C_{n-j}^{(-j-1)}=-G_{n+2}\quad (n\geq 0).
\end{align}

In this note we present a combinatorial proof of the identity \eqref{pBc-id}.

Recently, Matsusaka \cite{M20} showed that both \eqref{pB-id} and \eqref{pBc-id} are special cases of a more general polynomial identity between symmetrized poly-Bernoulli numbers and Gandhi-polynomials. 

%%%%%%%%%%%%%%%%%%%%%
\section{Definitions}
%%%%%%%%%%%%%%%%%%%%%

In this section we give the necessary definitions and notation. To prove the identity, we establish a connection between \emph{barred Callan sequences} and \emph{Dumont permutations of the first kind} defined in~\cite{Dumont}. 

In the entry \href{http://oeis.org/A099594}{A099594} of the OEIS~\cite{OEIS}, Callan defines a permutation class that are enumerated by the poly-Bernoulli numbers.  These are essentially the same as the Callan sequences that were defined in \cite{BN18}. Based on this notion, the authors introduced \emph{barred Callan sequences} in \cite{BM20}. Here we recall these definitions. 

Let $k,n\geq 0$ and consider the sets $K=\{\kek{1},\ldots, \kek{k}\}\cup\{\kek{*}\}$ (referred to as {\it blue elements}), and $N=\{\piros{1},\ldots, \piros{n}\}\cup\{\piros{*}\}$ (referred to as {\it red elements}).

\begin{definition}
 A {\it Callan sequence} of size $k\times n$ is a sequence
 \[
   (\kek{B_1},\piros{R_1})(\kek{B_2},\piros{R_2})\cdots(\kek{B_m},\piros{R_m}) (\kek{B^*},\piros{R^*}),
 \] 
 for some $m$ with $0 \leq m \leq n$, such that:
 \begin{itemize}
  \item $\{\kek{B_1},\dots,\kek{B_m},\kek{B^*}\}$ form a set partition of $K$ into $m+1$ non-empty blocks,
  \item $\{\piros{R_1},\dots,\piros{R_m},\piros{R^*}\}$ form a set partition of $N$ into $m+1$ non-empty blocks,
  \item $\kek{*}\in \kek{B^*}$, and $\piros{*}\in \piros{R^*}$.
 \end{itemize}
\end{definition}

Note that a Callan sequence can be encoded in an obvious way by the following data:
\begin{itemize}
 \item an integer $m$ with $0 \leq m \leq \min(k,n)$,
 \item an ordered partition $(\kek{B_1},\dots,\kek{B_m})$ of some subset $K'\subset \{\kek{1},\dots,\kek{k}\}$,
 \item an ordered partition $(\piros{R_1},\dots,\piros{R_m})$ of some subset $N'\subset \{\piros{1},\dots,\piros{n}\}$.
\end{itemize}
This is done by letting $K' = K\setminus \kek{B^*}$ and $N' = N \setminus \piros{R^*}$.
 
To deal with these objects, it will be convenient to use the following terminology:
\begin{itemize}
 \item $\kek{B^*}$ and $\piros{R^*}$ are called \emph{extra blocks}, while $\kek{B_i}$ and $\piros{R_i}$ are called \emph{ordinary blocks},
 \item each pair $(\kek{B_i},\piros{R_i})$ or $(\kek{B^*},\piros{R^*})$ is a {\it Callan pair}, moreover, the former are called {\it ordinary pairs} and the latter is called the {\it extra pair}.
\end{itemize}

A Callan sequence is thus a linear arrangement of Callan pairs, with the extra pair at the end.

\begin{example}The Callan sequences with $k=2$ and $n=2$ are:
	\begin{align*}
		&(\kek{12*},\piros{12*}), & &(\kek{12},\piros{12})(\kek{*},\piros{*}), & &(\kek{1},\piros{12})(\kek{2*},\piros{*}), & &(\kek{2},\piros{12})(\kek{1*},\piros{*}), & &(\kek{12},\piros{1})(\kek{*},\piros{2*}),\\
		  &(\kek{12},\piros{2})(\kek{*},\piros{1*}), & &(\kek{1},\piros{1})(\kek{2*},\piros{2*}), & &(\kek{2},\piros{1})(\kek{1*},\piros{2*}), & &(\kek{1},\piros{2})(\kek{2*},\piros{1*}), & &(\kek{2},\piros{2})(\kek{1*},\piros{1*}),\\
		    &(\kek{1},\piros{1})(\kek{2},\piros{2})(\kek{*},\piros{*}), & &(\kek{1},\piros{2})(\kek{2},\piros{1})(\kek{*},\piros{*}), & &(\kek{2},\piros{1})(\kek{1},\piros{2})(\kek{*},\piros{*}), & &(\kek{2},\piros{2})(\kek{1},\piros{1})(\kek{*},\piros{*}).
	\end{align*}
\end{example}

\begin{definition}
 \emph{A barred Callan sequence} of size $k \times n$ is a sequence obtained by inserting a bar (denoted $|$) into a Callan sequence of size $k\times n$, with the restriction that the bar cannot be at the end of the sequence.  We denote by $\mathcal{C}_n^k$ the set of barred Callan sequences with $k$ blue elements and $n$ red elements.
\end{definition}

\begin{example}
 The $7$ barred Callan sequences for $k=2$ and $n=1$ are
\begin{align*}
&|(\kek{12*},\piros{1*}); \quad
|(\kek{12},\piros{1})(\kek{*},\piros{*}); \quad |(\kek{2},\piros{1})(\kek{1}\kek{*},\piros{*});\quad
|(\kek{1},\piros{1})(\kek{2}\kek{*},\piros{*});\\
& (\kek{1},\piros{1})|(\kek{2}\kek{*},\piros{*}); \quad
 \quad (\kek{2},\piros{1})|(\kek{1}\kek{*},\piros{*});\quad(\kek{12},\piros{1})|(\kek{*},\piros{*}) .
\end{align*}
\end{example}

It is shown in \cite{BM20} that the number of barred Callan sequences of size $k\times n$ is the $C$-poly-Bernoulli number $C_n^{(-k-1)}$.  To avoid negative indices in what follows, we denote $C_n^{k} = C_n^{(-k-1)}$.  In particular, we have $C_n^k=|\mathcal{C}_n^k|$.  

Table~\ref{C-numbers} gives the first few values of the numbers $C_n^k$.  Note that the symmetry $C_n^k = C_k^n$ is clear from the combinatorial interpretation in terms of barred Callan sequences, as one can exchange red and blue blocks to get a $n\times k$ Callan sequence from a $k\times n$ Callan sequence.

\begin{center}
\begin{table}[ht]
\begin{tabular}{c|cccccc}
&0&1&2&3&4&5\\\hline
0&1&1&1&1&1&1\\
1&1&3&7&15&31&63\\
2&1&7&31&115&391&1267\\
3&1&15&115&675&3451&16275\\
4&1&31&391&3451&25231&164731\\
5&1&63&1267&16275&164731&1441923\\
\end{tabular}
\medskip
\caption{$C$-poly-Bernoulli numbers $C_n^k$.\label{C-numbers}}
\end{table}
\end{center}

\begin{remark}
We note here that Callan sequences are equivalent to the special class of partial permutations where elements greater than $n$ are excedances and the elements smaller than or equal to $n$ are deficiencies.  The number $C_n^k$ also counts alternative tableaux of rectangular shape $k\times n$, and other tableaux of rectangular shapes that are in bijection with alternative tableaux such as permutation tableaux or tree-like tableaux.  See~\cite{BH15,BH17} for details.
\end{remark}

The other number sequence that is involved in the identity we want to prove is the Genocchi numbers.  Genocchi numbers can be defined in different ways, here we focus on the combinatorial point of view. There are several combinatorial objects that are enumerated by the Genocchi numbers. The first known interpretation was Dumont permutations of the first kind. 

\begin{definition}[\cite{Dumont}]
 A \emph{Dumont permutation of the first kind} is a permutation $\pi\in S_{2n}$, such that each even entry begins a descent and each odd entry begins an ascent or ends the string, i.e., for every $i=1,2,\ldots, 2n$:
\begin{align*}
\pi(i) \text{ even }\quad &\Longrightarrow  \quad i<2n \text{ and } \pi(i)>\pi(i+1),\\
\pi(i) \text{ odd }\quad &\Longrightarrow \quad  i<2n \text{ and } \pi(i)<\pi(i+1),  \mbox{ or } i=2n.
\end{align*}
Let $\mathcal{D}_{2n}$ denote the set of Dumont permutations on $2n$ elements.
\end{definition}

\begin{example}
For $2n=4$ we have $\mathcal{D}_{4}=\{2143, 3421, 4213\}$. For $2n=6$ we have 
\begin{align*}
\mathcal{D}_{6}=\{642135, 634215,  621435
421365,342165,214365,564213,563421,562143,\\
216435,435621,215643,436215,364215, 421563,356421,421635.\}
\end{align*}
\end{example}

In general, we have the following:

\begin{theorem}[Dumont~\cite{Dumont}]
  The cardinality of $\mathcal{D}_{2n}$ is $(-1)^{n+1}G_{2n+2}$.
\end{theorem}

%%%%%%%%%%%%%%%%%%%%%
\section{Main result}
%%%%%%%%%%%%%%%%%%%%%

Using our notation, Identity~\eqref{pBc-id} can be reformulated as follows:

\begin{theorem}[\cite{KST18}] \label{theo_identity}
  For $n\geq 0$, it holds that:
  \begin{align}\label{identity0}
  \sum_{j=0}^n (-1)^jC_{n-j}^j= -G_{n+2}.
  \end{align}
\end{theorem}

We introduce some combinatorial objects, that will be at the core of our combinatorial proof of a generalization of the identity above.  They can be seen as barred Callan sequences with extra structure.  We could say, that they interpolate in a certain sense between barred Callan sequences and Dumont permutations.  

\begin{definition}
  Let $\mathcal{C}_n^k(m)$ denote the set of sequences defined by the following conditions.  First, the elements in the sequence are:
  \begin{itemize}
    \item $m$ blue bars labelled with $1,\dots,m$, and $m+1$ red bars labelled with $0,\dots,m$,
    \item the Callan pairs of a $k \times n$ Callan sequence, shifted up by $m$ (i.e., with base sets $K = \{\kek{m+1}, ..., \kek{m+k}\} \cup \{\kek{*}\}$ and $N = \{\piros{m+1}, ..., \piros{m+n}\} \cup \{\piros{*}\}$).
    \end{itemize}
    Second, these elements (labelled bars and Callan pairs) are ordered in a way such that:
    \begin{itemize}
    \item A blue bar with label $i$ is followed by a bar with label strictly smaller than $i$.
    \item A red bar with label $i$ is followed either by a Callan pair, or by a bar with label strictly greater than $i$.
  \end{itemize}
  The elements of $\mathcal{C}_n^k(m)$ are called $m$-{\it barred Callan sequences}, and the cardinality of $\mathcal{C}_n^k(m)$ is denoted $C_n^k(m)$.
\end{definition}

\begin{remark}
  The $m$-barred Callan sequences are defined otherwise in \cite{BM20}, we call our special objects $m$-barred Callan sequence for the sake of convenience.  
\end{remark}

\begin{remark}
  The numbers $C_n^k(1)$ are the so-called {\it poly-Bernoulli} $D$-{\it relatives}, also introduced in \cite{BH17} as the number of lonesum matrices that do not contain any all-zero column and any all-zero row. 
\end{remark}

Note that a bar is always followed by another element in the sequence, so the last element of the sequence is a Callan pair (more precisely, it is necessarily the extra pair of the underlying Callan sequence).  For example, an element of $\mathcal{C}_4^6(3)$ is:
\[
   \kek{|_3|_2}\piros{|_1|_2}               (\kek{5},  \piros{45}) 
               \piros{|_3}                  (\kek{79}, \piros{7})  
                                            (\kek{4},  \piros{6})
   \kek{|_1}   \piros{|_0}                  (\kek{68*},\piros{*}).
\]

\begin{note} \label{remarks_groups}
Sometimes it is convenient to think that each Callan pair is associated to the (possibly empty) sequence of consecutive bars preceding it.  In the previous example, we have the four ``groups'':
\[
   \kek{|_3|_2}\piros{|_1|_2}               (\kek{5},  \piros{45})
   \quad 
               \piros{|_3}                  (\kek{79}, \piros{7})
   \quad
                                            (\kek{4},  \piros{6})
   \quad
   \kek{|_1}   \piros{|_0}                  (\kek{68*},\piros{*}).
\]
Moreover, the last bar of each group is a red bar.
\end{note}

Let us first record some properties that are simple consequences of the definition.

\begin{proposition} \label{Cnkm_identities}
We have the following identities:
\begin{align}
  C_n^k(m) &= C_k^n(m),  \\
  C_n^k(0) &= C_n^k,     \\
  C_0^0(m) &= (-1)^{m+1} G_{2m+2}.
\end{align}
\end{proposition}

\begin{proof}
There is a simple bijection between $\mathcal{C}_n^k(m)$ and $\mathcal{C}_k^n(m)$: it consists in exchanging blue blocks with red blocks in each Callan pair.  As a consequence, we get the first identity.

The elements in $\mathcal{C}_k^n(0)$ contain a unique bar, $\piros{|_0}$. We thus have a simple bijection between $\mathcal{C}_k^n(0)$ and $\mathcal{C}_k^n$ by changing $\piros{|_0}$ into $|$.  Hence, we get the second identity.

The elements in $\mathcal{C}_0^0(m)$ contain a unique Callan pair, namely $(\kek{*},\piros{*})$, in last position.  This Callan pair is preceded by $\piros{|_m}$, as there is no label strictly greater than $m$.  Now, there is a simple bijection between $\mathcal{C}_0^0(m)$ and $\mathcal{D}_{2m}$: remove $\piros{|_m} (\kek{*},\piros{*})$, transform $\kek{|_i}$ into $2i$ and $\piros{|_i}$ into $2i+1$.
\end{proof}

Our main result is that the numbers $C_n^k(m)$ also have an alternating sum equal to Genocchi numbers:

\begin{theorem}\label{theo_identity2}
  For $n,m\geq 0$, we have:
  \begin{align}\label{identity}
  \sum_{j=0}^{n} (-1)^j C_{n-j}^j(m) = (-1)^{m+1} G_{n+2m+2}.
  \end{align}
\end{theorem}

Note that when $n$ is odd, both sides of~\eqref{identity} are $0$ (the sum has an even number of terms, and they can be paired using the symmetry $C_n^k(m) = C_k^n(m)$).  So only the case where $n$ is even is of interest.  Remark that the case $m=0$ in the previous theorem is precisely Theorem~\ref{theo_identity}.

We will use the following identity:
 
\begin{proposition} \label{prop_rec}
  For $n\geq 2$ and $m\geq 0$, we have:
  \begin{align}\label{identity2}
    \sum_{j=0}^{n} (-1)^j C_{n-j}^j(m) 
    = 
    - \sum_{j=0}^{n-2} (-1)^j C_{n-2-j}^j(m+1) 
  \end{align}
\end{proposition}

The combinatorial proof of Proposition~\ref{prop_rec} will be done in the next two sections.  To finish the present section, let us prove that Theorem~\ref{theo_identity2}, and therefore Theorem~\ref{theo_identity}, are consequences of Proposition~\ref{prop_rec}.

\begin{proof}[Proof of Theorem~\ref{theo_identity2}]
As noted above, we can assume $n$ is even.  Denote $n=2n'$.  We can iterate the identity in Proposition~\ref{prop_rec} successively with $(n,m)$, then $(n-2,m+1)$, etc.  Each iteration gives a sum with one less term, and we eventually get:
\[
  \sum_{j=0}^{n} (-1)^j C_{n-j}^j(m) 
  =
  (-1)^{n'} C_0^0( m + n' ).
\]
From the third equation in Proposition~\ref{Cnkm_identities}, this is also $(-1)^{n'+n'+m+1} G_{ 2m+2n'+2 }$ which simplifies to $(-1)^{m+1} G_{n+2m+2}$.
\end{proof}

%%%%%%%%%%%%%%%%%%%%%%%%%%%%%
\section{\texorpdfstring{The bijection $\phi$}{The bijection φ}}
%%%%%%%%%%%%%%%%%%%%%%%%%%%%%

Here we use the map introduced in \cite{BH15} in order to prove~\eqref{pB-id}.  We slightly modify this map, to reformulate it on the set of $m$-barred Callan sequences.  This map is a bijection between two particular subsets of $m$-barred Callan sequences.

\begin{definition}
Let $\mathcal{C}_n^k(m,\piros{*})$ denote the subset of $\mathcal{C}_n^k(m)$ containing elements $\alpha$ such that the extra red block is empty (by convention, this means it contains only $\piros{*}$).  Let $\mathcal{C}_n^k(m,\piros{R*}) = \mathcal{C}_n^k(m) \backslash \mathcal{C}_n^k(m,\piros{*})$ denote the complementary subset.

Following the previous convention, let $C_n^k(m,\piros{*})$ (respectively, $C_n^k(m,\piros{R*})$) denote the cardinality of $\mathcal{C}_n^k(m,\piros{*})$ (respectively, $\mathcal{C}_n^k(m,\piros{R*})$).
\end{definition}

Suppose that $n\geq 1$ and $k\geq0$ are fixed throughout this section. We define the map 

\[
  \phi: \mathcal{C}_n^k(m,\piros{R*})\rightarrow \mathcal{C}_{n-1}^{k+1}(m,\piros{*})
\] 

as follows.  Let $\alpha$ be an $m$-barred Callan sequence with a non-empty extra red block, $\alpha\in \mathcal{C}_n^k(m,\piros{R*})$.  The general idea is to remove the maximal red element $\piros{m+n}$ and add a new blue element $\kek{m'} = \kek{m+k+1}$.  Distinguish four cases according to the location of the maximal red element $\piros{m+n}$:
\begin{itemize}
  \item[A1)] $\piros{m+n}$ is in the extra block as a singleton, $\piros{R^*}=\{\piros{m+n,*}\}$. 
  \item[A2)] $\piros{m+n}$ is in the extra block with other red elements, $\piros{R^*}\setminus \{\piros{m+n}\}=\piros{R'}\not=\{\piros{*}\}$. 
  \item[B1)] $\piros{m+n}$ is in an ordinary block as a singleton, i.e., there is an ordinary pair $(\kek{B_i},\piros{R_i})$ with $\piros{R_i} = \{\piros{m+n}\}$. 
  \item[B2)] $\piros{m+n}$ is in an ordinary block with other red elements, i.e., there is an ordinary pair $(\kek{B_i}, \piros{R_i})$ with $\piros{m+n}\in\piros{R_i}$ and $\piros{R_i}\setminus \{\piros{m+n}\}=\piros{R'}\neq\varnothing$.  
\end{itemize}
Accordingly, $\phi(\alpha)$ is defined as the result of the following procedure:
\begin{itemize}
  \item[A1)] Delete $\piros{m+n}$ and add the new element $\kek{m'}$ to the first Callan pair. 
  \item[A2)] Replace the extra red block $\piros{R^*}$ with $\{\piros{*}\}$, and place $(\{\kek{m'}\},\piros{R'})$ in front of the Callan sequence as the first element. 
  \item[B1)] The successive steps are as follows:
    \begin{itemize}
      \item replace the extra red block $\piros{R^*}$ with $\{\piros{*}\}$,
      \item replace $\piros{R_i} = \{\piros{m+n}\}$ with $\piros{R^*}\backslash\{\piros{*}\}$,
      \item move the new pair $(\kek{B_i},\piros{R^*}\backslash\{\piros{*}\})$ in front of the Callan sequence, together with the group of bars to its left (as in Note~\ref{remarks_groups}),
      \item add $\kek{m'}$ to the $(i+1)$st blue block (this means either $\kek{B_{i+1}}$, or the extra blue block $\kek{B^*}$ if $\kek{B_i}$ was in the last ordinary pair). 
    \end{itemize}

  \item[B2)] The several steps are as follows:
    \begin{itemize}
      \item replace the extra red block $\piros{R^*}$ with $\{\piros{*}\}$,
      \item replace $(\kek{B_i},\piros{R_i})$ with two pairs $(\kek{B_i},\piros{R^*}\backslash\{\piros{*}\})$ and $(\{\kek{m'}\},\piros{R'})$,
      \item move the new pair $(\kek{B_i},\piros{R^*}\backslash\{\piros{*}\})$ in front of the sequence, together with the group of consecutive bars preceding it (as in Remark~\ref{remarks_groups}). 
    \end{itemize}
\end{itemize}

\begin{example}
Suppose that $m, n, k$ are such that we remove $\piros{9}$ and add $\kek{8}$.  The four cases can be illustrated as follows:

\begin{align*}
A1) \quad 
(\kek{5}, \piros{35})\piros{|_1}(\kek{47},\piros{247})\kek{|_1}\piros{|_0}(\kek{23},\piros{68})(\kek{6*},\piros{9*})
\quad&\rightarrow\quad
(\kek{58}, \piros{35})\piros{|_1}(\kek{47},\piros{247})\kek{|_1}\piros{|_0}(\kek{23},\piros{68})(\kek{6*},\piros{*})\\
A2) \quad 
(\kek{5}, \piros{35})\piros{|_1}(\kek{47},\piros{2})\kek{|_1}\piros{|_0}(\kek{23},\piros{68})(\kek{6*},\piros{479*})
\quad&\rightarrow\quad
(\kek{8},\piros{47})(\kek{5}, \piros{3,5})\piros{|_1}(\kek{47},\piros{2})\kek{|_1}\piros{|_0}(\kek{23},\piros{68})(\kek{6*},\piros{*})\\
B1) \quad 
(\kek{5}, \piros{358})\piros{|_1}(\kek{47},\piros{24})\kek{|_1}\piros{|_0}(\kek{23},\piros{9})(\kek{6*},\piros{67*})
\quad&\rightarrow\quad
\kek{|_1}\piros{|_0}(\kek{23},\piros{67})(\kek{5}, \piros{358})\piros{|_1}(\kek{47},\piros{24})(\kek{68*},\piros{*})\\
B2) \quad 
(\kek{5}, \piros{38})\piros{|_1}(\kek{47},\piros{4})\kek{|_1}\piros{|_0}(\kek{23},\piros{259})(\kek{6*},\piros{67*})
\quad&\rightarrow\quad
\kek{|_1}\piros{|_0}(\kek{23},\piros{67})(\kek{5}, \piros{38})\piros{|_1}(\kek{47},\piros{4})(\kek{8},\piros{25})(\kek{6*},\piros{*})
\end{align*}

\end{example}

\begin{definition} \label{defi_barm}
  For $n\geq0$ and $k\geq1$, let $\mathcal{C}_n^k(m,\piros{*},\kek{|m'})$ denote the subset of $\mathcal{C}_n^k(m,\piros{*}) $, containing elements $\alpha$ such that the maximal blue element $\kek{m'} = \kek{m+k}$ is in an ordinary Callan pair as a singleton, and there is at least one bar preceding this pair.  Following the previous convention, let $C_n^k(m,\piros{*},\kek{|m'})$ denote the cardinality of $\mathcal{C}_n^k(m,\piros{*},\kek{|m'})$.
\end{definition}

Note that when $k=0$ and $n>0$, we have $\mathcal{C}_n^0(m,\piros{*}) = \varnothing $.  Indeed, the only $0\times n$ Callan sequence is $(\kek{*},\piros{123\dots n*})$, which has a nonempty extra red block.  Accordingly, we also take the convention $\mathcal{C}_n^0(m,\piros{*},\kek{|m'}) = \varnothing$.

Note also that $\mathcal{C}_0^k(m,\piros{*},\kek{|m'}) = \varnothing$.  Indeed, the only $k\times 0$-Callan pair is $(\kek{123\dots k*},\piros{*})$, so it is not possible to have $\kek{m'}$ as a singleton in an ordinary pair.

\begin{lemma}
The map $\phi$ defined as above is a bijection between the sets $\mathcal{C}_n^k(m,\piros{R*})$ and $\mathcal{C}_{n-1}^{k+1}(m,\piros{*})\setminus \mathcal{C}_{n-1}^{k+1}(m,\piros{*},\kek{|m'})$. 
\end{lemma}
\begin{proof}

Let us first check that $\phi(\alpha) \in \mathcal{C}_{n-1}^{k+1}(m,\piros{*})$.  Note that each group of consecutive bars is unchanged (though these groups are possibly permuted in case B1 and B2), so we only need to check that the pairs form a valid Callan sequence.  This is straightforward.

Then, let us check that $\phi(\alpha) \notin \mathcal{C}_{n-1}^{k+1}(m,\piros{*},\kek{|m'})$.  In cases A1 and B1, the blue element $\kek{m'}$ is added to a blue block (which is nonempty by definition) so  $\kek{m'}$ is not a single element in its block.  In cases A2 and B2, $\kek{m'}$ is in a singleton block, so we need to check that there is no bar preceding the pair containing $\kek{m'}$.  In case A2, this is because this pair is at the beginning of the sequence.  In case B2, the new pair $(\{\kek{m'}\},\piros{R'})$ is created with another pair to its left, and since the pair to its left is moved together with the associated group of bars, there is again another pair to the left of $(\{\kek{m'}\},\piros{R'})$ in $\phi(\alpha)$.

In the previous paragraph, we observed the following properties in $\phi(\alpha)$:
\begin{itemize}
 \item $\kek{m'}$ is in the first pair iff the sequence was obtained from the cases A1 or A2. 
 \item $\kek{m'}$ is alone in its block iff the sequence was obtained from the cases A2 or B2.
\end{itemize}
So, the four cases can be 
distinguished based on which of the above properties they satisfy.  In each case, knowing the position of $\kek{m'}$ it is straightforward to describe the inverse procedure, so that $\phi$ is injective.  
\end{proof}

The bijection $\phi$ can be used to obtain pairwise cancellations in the alternating sums, and we get the following:

\begin{lemma}
 We have:
 \[
   \sum_{j=0}^n (-1)^j C_{n-j}^j(m)
   =
   \sum_{j=0}^{n} (-1)^{j} C_{n-j}^j(m,\piros{*},|\kek{m'}).
 \]
\end{lemma}

\begin{proof}
The set $\mathcal{C}_{n-j}^j(m)$ can be partitioned as
\begin{align}\label{split}
 \mathcal{C}_{n-j}^j(m) 
 = 
   \mathcal{C}_{n-j}^j(m,\piros{R*}) 
 \biguplus 
   \Big(\mathcal{C}_{n-j}^j(m,\piros{*}) \setminus \mathcal{C}_{n-j}^j (m,\piros{*},|\kek{m'})\Big) 
 \biguplus 
   \mathcal{C}_{n-j}^j (m,\piros{*},|\kek{m'}),
\end{align}
so that 
\begin{align*}
 \sum_{j=0}^n (-1)^j C_{n-j}^j(m)
 =
 \sum_{j=0}^n (-1)^j|\mathcal{C}_{n-j}^j(m,\piros{R*})|
  +
 \sum_{j=0}^n (-1)^j &\Big|\Big(\mathcal{C}_{n-j}^j(m,\piros{*}) \setminus \mathcal{C}_{n-j}^j (m,\piros{*},|\kek{m'})\Big)\Big| \\
  &+
 \sum_{j=0}^n (-1)^j C_{n-j}^j (m,\piros{*},|\kek{m'}).
\end{align*}
The bijection $\phi$ readily shows that the first and second sum cancel each other out, upon checking the boundary terms.  

The first boundary term is $j=n$ in the first sum.  The set $\mathcal{C}_0^n(m,\piros{R*})$ is empty (as there are no red element besides $\piros{*}$, the extra red block cannot be nonempty), so the corresponding term is $0$.  The second boundary term is $j=0$ in the second sum.  As noted after Definition~\ref{defi_barm}, $\mathcal{C}_n^{0}(m,\piros{*})$ is empty,
$\mathcal{C}_n^{0}(m,\piros{*})\setminus \mathcal{C}_n^{0}(m,\piros{*},|\kek{m'})$ is empty as well, and the corresponding term is $0$.
\end{proof}

Using the previous lemma to prove Proposition~\ref{prop_rec} it remains only to show:
\begin{align} \label{id-var1}
  \sum_{j=0}^{n} (-1)^{j} C_{n-j}^j(m,\piros{*},|\kek{m'})
  =
  - \sum_{j=0}^{n-2} (-1)^j C_{n-2-j}^j(m+1) .
\end{align} 
This will be done in the next section.

%%%%%%%%%%%%%%%%%%%%%%%%%%%%
\section{\texorpdfstring{The bijection $\psi$}{The bijection ψ}}
%%%%%%%%%%%%%%%%%%%%%%%%%%%%

To prove \eqref{id-var1}, it is convenient to use another set. 

\begin{definition}
We define $\mathcal{C}_{n}^k(m,\piros{*},|\kek{m+1})$ like $\mathcal{C}_{n}^k(m,\piros{*},|\kek{m'})$ in Definition~\ref{defi_barm}, but with $m+1$ instead of $m'=m+k$. 
As before, denote $C_{n}^k(m,\piros{*},|\kek{m+1}) = |\mathcal{C}_{n}^k(m,\piros{*},|\kek{m+1})|$.

Explicitely, $\mathcal{C}_n^k(m,\piros{*},\kek{|m'})$ is the subset of $\mathcal{C}_n^k(m,\piros{*}) $ containing elements $\alpha$ such that the minimal blue element $\kek{m+1}$ is in an ordinary Callan pair as a singleton, and there is at least one bar preceding this pair. 
\end{definition}

There is a bijection from $\mathcal{C}_{n}^k(m,\piros{*},|\kek{m'})$ to $\mathcal{C}_{n}^k(m,\piros{*},|\kek{m+1})$, via a relabelling of blue elements that exchanges $\kek{m+k}$ and $\kek{m+1}$.  Using this, \eqref{id-var1} is clearly equivalent to:
\begin{align} \label{id-var2} 
  \sum_{j=0}^{n} (-1)^{j} C_{n-j}^j(m,\piros{*},|\kek{m+1})
  =
  - \sum_{j=0}^{n-2} (-1)^j C_{n-2-j}^j(m+1) .
\end{align} 

In order to show \eqref{id-var2}, we give a bijection $\psi$ between the sets $\mathcal{C}_n^{k}(m,\piros{*},|\kek{m+1})$ and $\mathcal{C}_{n-1}^{k-1}(m+1)$ for $n,k\geq1$.  Note that $\mathcal{C}_0^{k}(m,\piros{*},|\kek{m+1})=\mathcal{C}_n^{0}(m,\piros{*},|\kek{m+1})=\varnothing$, following the remarks after Definition~\ref{defi_barm}.

We do this in two steps, i.e., as the composition of two operations: $\psi=\psi_r \circ \psi_b$.  The first map $\psi_b$ takes the blue element $\kek{m+1}$ out of a Callan pair and changes it into a labeled blue bar $\kek{|_{m+1}}$. The second map $\psi_r$ does the same with $\piros{m+1}$ and $\piros{|_{m+1}}$.

\begin{definition}
  For $\alpha \in \mathcal{C}_{n}^k(m,\piros{*},|\kek{m+1})$, note that $\alpha$ contains an ordinary pair $(\kek{m+1},\piros{R})$.  Let $w_1$ (respectively, $w_2$) denote the maximal subsequence of consecutive bars that are directly to the left (respectively, to the right) of $(\kek{m+1},\piros{R})$.  
  Note that $w_1$ is nonempty by definition, but $w_2$ is possibly empty.  We define $\psi_b(\alpha)$ as follows.
  \begin{itemize}
   \item Replace $w_1 \; (\kek{m+1},\piros{R}) \; w_2$ with $w_2 \; \kek{|_{m+1}} \; w_1$ in the sequence.
   \item Replace the extra red block $\{\piros{*}\}$ with $\piros{R}\cup\{\piros{*}\}$.
  \end{itemize}
\end{definition}

\begin{example}
\begin{align*}
(\kek{4},\piros{9})(\kek{58},\piros{25})\piros{|_1}(\kek{2},\piros{347})\kek{|_1}\piros{|_0}(\kek{37},\piros{68})(\kek{6*},\piros{*})\rightarrow_{\psi_b} 
(\kek{4},\piros{9})(\kek{58},\piros{35})\kek{|_1}\piros{|_0}\kek{|_2}\piros{|_1}(\kek{37},\piros{68})(\kek{6*},\piros{347*})
\end{align*}
\end{example}

Note that an element in the image of $\psi_b$ has a nonempty extra red block.

\begin{definition}
Let $\alpha \in \im (\psi_b)$.  We define $\psi_r(\alpha)$ by the following procedure,
distinguishing two cases:
\begin{itemize}  
\item If $\piros{m+1} \in \piros{R^*}$, replace the extra pair $(\kek{B^*}, \piros{R^*})$ with two consecutive elements $\piros{|_{m+1}}$ and $(\kek{B^*} , \piros{R^*}\setminus\{\piros{m+1}\})$.
\item Otherwise, $\piros{m+1}$ is in an ordinary pair $(\kek{B_i}, \piros{R_i})$.  Replace the pair $(\kek{B_i}, \piros{R_i})$ with two consecutive elements $\piros{|_{m+1}}$ and $(\kek{B_i},\piros{R^*}\setminus\{\piros{*}\})$, then replace the extra pair $(\kek{B^*}, \piros{R^*})$ with $(\kek{B^*}, \piros{R'}\cup\{\piros{*}\})$ where $\piros{R'}=\piros{R_i}\setminus\{\piros{m+1}\}$.
\end{itemize}
\end{definition}

\begin{example}
\begin{align*}
(\kek{583},\piros{46})(\kek{29},\piros{35})\kek{|_1}\piros{|_0}(\kek{47},\piros{2})(\kek{3*},\piros{17*})&\rightarrow_{\psi_r}
(\kek{583},\piros{46})(\kek{29},\piros{35})\kek{|_1}\piros{|_0}(\kek{47},\piros{2})\piros{|_1}(\kek{3*},\piros{7*})\\
 (\kek{583},\piros{1})(\kek{29},\piros{35})\kek{|_1}\piros{|_0}(\kek{47},\piros{46})(\kek{3*},\piros{27*})&\rightarrow_{\psi_r}
 \piros{|_1}(\kek{583},\piros{27})(\kek{29},\piros{35})\kek{|_1}\piros{|_0}(\kek{47},\piros{46})(\kek{3*},\piros{*})
\end{align*}
\end{example}

\begin{lemma}
For $n,k\geq1$ and $m\geq 0$, the map $\psi = \psi_r \circ \psi_b$ is a bijection between the sets $\mathcal{C}_n^{k}(m,\piros{*},|\kek{m+1})$ and $\mathcal{C}_{n-1}^{k-1}(m+1)$.
\end{lemma}

\begin{proof}
 First, we can check that $\psi(\alpha)$ with labelled bars removed is a $(k-1)\times(n-1)$ Callan sequence where labels are shifted up by $m+1$.  As noted above, $\psi_b(\alpha)$ has a nonempty extra red block.  In the second case of the definition of $\psi_r$, we thus see that $\piros{R^*}\setminus\{\piros{*}\}\neq\varnothing$ so that $(\kek{B_i},\piros{R^*}\setminus\{\piros{*}\})$ is a valid Callan pair.  
 
 Secondly, we check that the labelled bars in $\psi(\alpha)$ sastisfy the conditions so that $\psi(\alpha)\in \mathcal{C}_{n-1}^{k-1}(m+1)$.  We begin with labelled bars in $\psi_b(\alpha)$.  There are two locations to check:
 \begin{itemize}
  \item The element after $\kek{|_{m+1}}$ should be a labelled bar with label strictly smaller than $m+1$.  This element is the first one of $w_1$, as $w_1$ is nonempty.  By definition, it satisfies the required properties.
  \item If $w_2$ is empty, either $\kek{|_{m+1}}$ begins the sequence or it is preceded by a Callan pair.  Otherwise, the last element of $w_2$ is a bar $\piros{|_i}$ with $i\leq m$, as it is followed by a Callan pair in $\alpha$.  In both cases, this is a valid configuration.
\end{itemize}
Then, $\psi(\alpha)$ is obtained by adding $\piros{|_{m+1}}$ to the left of a Callan pair.  Thus, either it is a new group of bars in itself, or it is added at the end of an existing group of bars.  In the second case, as the last element of this group is a red bar with label $i$ at most $m$, the configuration $\piros{|_i |_{m+1}}$ is valid.

We thus have $\psi(\alpha) \in \mathcal{C}_{n-1}^{k-1}(m+1)$.  Now it remains to describe the inverse bijection $\psi^{-1}$ and check that it is indeed the left and right inverse of $\psi$. 
In the definition of $\psi_r$, the cases can be distinguished: the first one (respectively, second one) results in $\piros{|_{m+1}}$ being to the left of the extra pair (respectively, of an ordinary pair).  Using that, the inverse bijection is straightforward to describe explicitly.
\end{proof}

\begin{remark}
	Our bijections provided a combinatorial proof for the fact that the alternating diagonal sum of $C$-poly-Bernoulli numbers is given by the Genocchi number (Theorem~\ref{theo_identity}). We used a special technique: instead of defining a direct involution on the set that is enumerated by the $C$-poly-Bernoulli numbers, we introduced a kind of interpolating sets, the $m$-barred Callan sequences. Starting with the set of Callan sequences, enumerated by the $C$-poly-Bernoulli number, we applied the involution $\phi$ that reduced the number of elements to obtain a new set, on which we again applied the involution, until we reached the set that is known to be counted by the Genocchi number. This tricky solution shows somewhat the difficulty of the problem. It would be also interesting to find a direct bijection on any of the corresponding combinatorial objects that proves Theorem~\ref{theo_identity} without any iteration steps needed. Moreover, it is still an open problem to find a combinatorial proof for the more general identity related with symmetrized poly-Bernoulli numbers and Gandhi polynomials (see~\cite{M20}).
\end{remark}

\section*{Acknowledgements}
The authors would like to thank the anonymous referees for helpful comments and suggestions.


\begin{thebibliography}{99}
	
	\bibitem{AK99A}	\textsc{T. Arakawa} and \textsc{M. Kaneko},
        Multiple Zeta values, poly-Bernoulli numbers, and related zeta functions,
        \textit{Nagoya Math. J.} \textbf{153} (1999), 189--209.
	
	\bibitem{AK99} \textsc{T.~Arakawa} and \textsc{M.~Kaneko},
		On poly-Bernoulli numbers,
		\textit{Comment. Math. Univ. St. Paul.} \textbf{48} (1999), 159--167.
		
	\bibitem{BH15} \textsc{B.~B\'enyi} and  \textsc{P.~Hajnal}, 
		Combinatorics of poly-Bernoulli numbers,
		\textit{Studia Sci.~Math.~Hungarica} \textbf{52} (2015), 537--558.
		
	\bibitem{BH17} \textsc{B.~B\'{e}nyi} and \textsc{P.~Hajnal}, 
		Combinatorial properties of poly-Bernoulli relatives,
		\textit{Integers} \textbf{17} (2017), A31.
		
	\bibitem{BM20} \textsc{B.~B\'{e}nyi} and \textsc{T.~Matsusaka}, 
        On the combinatorics of symmetrized poly-Bernoulli numbers.
        \textit{Electron.~J.~Combin.~} \textbf{28(1)}  (2021), P1.47.
        
	\bibitem{BN18} \textsc{B.~B\'{e}nyi} and \textsc{G.~Nagy},
        Bijective enumeration of $\Gamma$-free 0-1 matrices,
        \textit{Adv. in Appl. Math.} \textbf{96} (2018), 195--215.
	
	\bibitem{Brew} \textsc{C. R. Brewbaker},
        Lonesum $(0,1)$ matrices and the poly-Bernoulli numbers of negative index.
        Master's thesis, Iowa State Univ., 2005
        (\url{https://lib.dr.iastate.edu/rtd/18914/}).

    \bibitem{Dumont} \textsc{D. Dumont},
        Interpr\'etations combinatoires des nombres de Genocchi,
        \textit{Duke Math. J.} \textbf{41} (1974), 305--318.

    \bibitem{K97} \textsc{M.~Kaneko},
		Poly-Bernoulli numbers,
		\textit{J. Th\'{e}or. Nombres Bordeaux} \textbf{9} (1997), 221--228.
		
	\bibitem{KST18} \textsc{M. Kaneko, F. Sakurai} and \textsc{H. Tsumura},
		On a duality formula for certain sums of values of poly-Bernoulli polynomials and its application.
		\textit{J. de Th\'{e}orie des Nombres de Bordeaux}, 30-1, (2018) 203--218.
		
	\bibitem{M20} \textsc{T.~Matsusaka},
		Symmetrized poly-Bernoulli numbers and combinatorics,
		\textit{J. Integer Seq.} \textbf{23(9)} (2020), A20.9.2.
		
	\bibitem{OEIS} \textsc{N.~J.~A.~Sloane},
        The on-line encyclopedia of integer sequences, \url{http://oeis.org}
	
	\bibitem{Stanley} \textsc{R.~P.~Stanley},
	\textit{Enumerative Combinatorics}, Vol 2. Cambridge University Press, 1999. 
	
\end{thebibliography}
\end{document}